\documentclass[12pt]{amsart}

\usepackage{amssymb,upref,enumerate}
\usepackage{tikz}
\usepackage{amsmath,amssymb,amscd,amsthm,amsxtra}
\usepackage{latexsym}
\usepackage{yfonts,mathrsfs, esint}
\usepackage{amsmath}
\usepackage{bbm}           
\usepackage{bm} 
\usepackage{eso-pic,graphicx}

\calclayout

\usepackage{pdfsync}

\usepackage[colorlinks=true, bookmarks=true,pdfstartview=FitV, linkcolor=blue, citecolor=blue, urlcolor=blue]{hyperref}

\def\Xint#1{\mathchoice 
{\XXint\displaystyle\textstyle{#1}}%
{\XXint\textstyle\scriptstyle{#1}}%
{\XXint\scriptstyle\scriptscriptstyle{#1}}%
{\XXint\scriptscriptstyle\scriptscriptstyle{#1}}%
\!\int} 
\def\XXint#1#2#3{{\setbox0=\hbox{$#1{#2#3}{\int}$} 
\vcenter{\hbox{$#2#3$}}\kern-.5\wd0}} 
 
\def\dashint{\Xint-}

\newcommand{\ra}{\rightarrow}

\newcommand{\bey}{\begin{eqnarray*}}
\newcommand{\eey}{\end{eqnarray*}}
\newcommand{\ba}{\begin{align}}
\newcommand{\ea}{\end{align}}
\newcommand{\bea}{\begin{align*}}
\newcommand{\ena}{\end{align*}}
\newcommand{\be}{\begin{equation}}
\newcommand{\ee}{\end{equation}}
\newcommand{\R}{\mathbb R}
\newcommand{\Z}{\mathbb Z}

\newcommand{\N}{\mathbb N}

\newcommand{\Sp}{\mathcal S }

\newcommand{\ep}{\epsilon}

\newcommand{\bc}{\begin{center}}
\newcommand{\ec}{\end{center}}

\newcommand{\al}{\alpha}

\newcommand{\supp}{\mathrm{supp}}

\DeclareMathOperator*{\sgn}{sgn}

\newtheorem{theorem}{Theorem}[section]
\newtheorem{lemma}[theorem]{Lemma}
\newtheorem{corollary}[theorem]{Corollary}

\theoremstyle{definition}

\theoremstyle{remark}

\numberwithin{equation}{section}
\allowdisplaybreaks

\begin{document}

\author{David Cruz-Uribe, OFS}
\address{David Cruz-Uribe, OFS \\ Department of Mathematics \\ University of Alabama \\
Tuscaloosa, AL 35487, USA}
\email{dcruzuribe@ua.edu}

\author{Kabe Moen}
\address{Kabe Moen \\Department of Mathematics \\ University of Alabama \\
  Tuscaloosa, AL 35487, USA}
\email{kabe.moen@ua.edu}

\author{Quan Minh Tran}
\address{Quan Minh Tran \\Department of Mathematics \\ University of Alabama \\
  Tuscaloosa, AL 35487, USA}
\email{qmtran@crimson.ua.edu}


\title[Bump conditions for commutators]{New oscillation classes and two weight bump conditions for commutators.}
\date{}

\begin{abstract}
In this paper we consider two weight bump conditions for higher order
commutators.  Given $b$ and a Calder\'on-Zygmund operator $T$, define
the commutator $T^1_bf=[T,b]f= bTf-T(bf)$, and for $m\geq 2$ define
the iterated commutator $T^m_b f = [b,T_b^{m-1}]f$.  Traditionally,
commutators are defined for functions $b\in BMO$, but we show that if
we replace $BMO$ by an oscillation class first introduced by
P\'erez~\cite{MR1317714}, we can give a range of  sufficient
conditions on a pair of weights $(u,v)$ for  $T^m_b : L^p(v)
\rightarrow L^p(u)$ to be bounded.  
Our results generalize work of
the first two authors in~\cite{MR2918187}, and more recent work by
Lerner, {\em et al.}~\cite{LORR}.  We also prove necessary conditions
for the iterated commutators to be bounded, generalizing results of
Isralowitz, {\em et al.}~\cite{IsPoTr}.
\end{abstract}

\thanks{The first author is supported by research funds from the Dean
  of the College of Arts \& Sciences, the University of Alabama. The
  second author is supported by Simons Collaboration Grant for
  Mathematicians, 160427. A version of this paper will serve as a chapter in
  the third author's Ph.D. dissertation. }

\maketitle

\section{Introduction}
\label{section:intro}

In this paper we will study the commutator 
$$[b,T]f(x)=b(x)Tf(x)-T(bf)(x),$$ 
where $T$ is a linear operator and $b$ is a function. If we let
$T^1_b=[b,T]$ and for $m\in \N$ define $T^m_b=[b,T^{m-1}_b]$, then
$T^m_b$ is referred to as an iterated commutator. We will consider $T$
to be a Calder\'on-Zygmund operator or CZO (see Section 2 for relevant
definitions).  It is well known that such operators are well behaved
when $b\in BMO$:
$$\|b\|_{BMO}=\sup_Q \dashint_Q |b(x)-b_Q|\,dx<\infty$$
where $b_Q=\dashint_Q b\,dx=\frac{1}{|Q|}\int_Qb(x)\,dx$. Coifman,
Rochberg, and Weiss \cite{CRW} showed that if $T$ is a
Calder\'on-Zygmund operator then $[b,T]$ is bounded on $L^p(\R^n)$ if
$b\in BMO$. We are interested in two weight inequalities of the form
\begin{equation*}\left(\int_{\R^n}|T^m_bf|^pu\right)^{\frac1p}
  \leq C\left(\int_{\R^n}|f|^pv\right)^{\frac1p}
\end{equation*}
where $(u,v)$ is a pair of weights. 

Two weight inequalities for CZOs are significantly more difficult than
the corresponding one weight inequalities. Sawyer was the first to
find a complete characterization for the Hardy-Littlewood maximal
operator and operators with positive kernels in \cite{Saw1,Saw2}.  He
proved that such operators are bounded from $L^p(v)$ to $L^p(u)$ if
and only if certain testing conditions are satisfied. Roughly
speaking, the testing conditions boil down to a restricted boundedness condition
on the family $\{{\mathbf 1}_Qv^{-\frac{p'}{p}}\}_Q$ and a similar
dual condition.  Recently, a complete characterization was found for
the Hilbert transform by Lacey, {\em et al.}~\cite{MR3285857,Lac1}.  While these testing
conditions are necessary and sufficient, they are difficult to verify
in practice since they involve the operator itself.

We will examine conditions known as \emph{bump conditions} on the pair of weights $(u,v)$ that are sufficient and sharp.  The natural two weight generalization of the one weight condition given by 
\begin{equation}\label{twoAp}
  \sup_Q\left(\dashint_Q u\right)^{\frac1p}
  \left(\dashint_Q v^{-\frac{p'}{p}}\right)^{\frac{1}{p'}}<\infty,
\end{equation}
is often necessary but almost never sufficient for
boundedness. However, it is possible to modify this condition to get a
``universal'' sufficient condition.   Condition \eqref{twoAp} can be written as
$$\sup_Q\|u^{\frac1p}\|_{p,Q}\|v^{-\frac1p}\|_{p',Q}<\infty$$
where $\|\cdot\|_{p,Q}$ and $\|\cdot\|_{p',Q}$ correspond to the $L^p$ and
$L^{p'}$ averages over $Q$ respectively.  If we enlarge or `bump' these
averages in the scale of Orlicz spaces, we can get two weight
boundedness of many classical operators.  To understand this `enlarging' we
need a few definitions. We say $A(t)$ defined on $[0,\infty)$ is a
Young function if it is increasing, convex, $A(0)=0$, and
$A(t)/t\ra \infty$ as $t\ra\infty$. We will also consider $A(t)=t$ to
be a Young function by convention. Given a Young function $A$ there
exists another Young function $\bar{A}$, called the associate function
of $A$, such that $\bar{A}^{-1}(t)A^{-1}(t)\approx t$. If $A$ is a
Young function, then we define the Orlicz average on a cube $Q$ of a function
$f$ by
\begin{equation}\label{eqn:orlicz}
  \|f\|_{A,Q}
  =\inf\left\{\lambda>0:\dashint_Q
    A\Big(\frac{|f(x)|}{\lambda}\Big)\,dx\leq 1\right\}.
\end{equation}
When $A(t)=t^p\log(e+t)^a$ for some $a\in \R$ we will write 
$$\|f\|_{A,Q}=\|f\|_{L^p(\log L)^a,Q}$$
and when $B(t)=\exp(t^a)-1$ we will write
$$\|f\|_{B,Q}=\|f\|_{\exp(L^a),Q}.$$
We say that a Young function $A$ belongs to $B_p$ if the following growth condition is satisfied:
\begin{equation}\label{Bp}\int_1^\infty\frac{A(t)}{t^p}\frac{dt}{t}<\infty.\end{equation}
A typical $B_p$ Young function is given by $A(t)=t^{p-\delta}$ or $A(t)=t^p(\log(e+t))^{-1-\delta}$ for some $\delta>0$.

The two weight bump theory for CZOs is fairly well developed.  Early
work on the bump conditions for CZOs was done by the first author and
P\'erez \cite{MR1793688,cruz-uribe-perez02}. They formulated the
so-called two weight bump conjecture in \cite{cruz-uribe-perez02},
conjecturing that if $T$ is a CZO, $A$ and $B$ are Young functions
with $\bar{A}\in B_{p'}$ and $\bar{B}\in B_p$, and the pair of weights
$(u,v)$ satisfies
\begin{equation}\label{bumpCZO}
  \sup_Q\|u^{\frac1p}\|_{A,Q}\|v^{-\frac1p}\|_{B,Q}<\infty,
\end{equation}
then 
\begin{equation*}\left(\int_{\R^n}|Tf|^pu\right)^{\frac1p}\leq C\left(\int_{\R^n}|f|^pv\right)^{\frac1p}.\end{equation*}
This conjecture was first shown for specific operators (e.g.~the
Hilbert transform, Riesz transforms, etc.) or in a restricted
range on $p$ \cite{CMP07,CMP09,MR3127385}. The full conjecture was
proved by Lerner in \cite{MR3127385}. In particular, he showed that
CZOs are bounded from $L^p(v)$ to $L^p(u)$ if $(u,v)$ satisfy
\eqref{bumpCZO} with $A(t)=t^p\log(e+t)^{p-1+\delta}$ and
$B(t)=t^{p'}\log(e+t)^{p'-1+\delta}$ (see Section~\ref{section:prelim} for the fact that
$\bar{A}\in B_{p'}$ and $\bar{B}\in B_p$).  Recent work on bump
conditions for CZOs involves separating the bump conditions, namely,
showing that the smaller condition
$$\|u^{\frac1p}\|_{A,Q}\left(\dashint_Q v^{-\frac{p'}{p}}\right)^{\frac{1}{p'}}+\left(\dashint_Q u\right)^{\frac1p}\|v^{-\frac1p}\|_{B,Q}<\infty$$
is sufficient for the boundedness of $T$. (For an example showing this
condition is weaker, see~\cite{MR3302574}.)  Determining the sharpest separated
bump conditions for CZOs is still an active area of research (see
\cite{MR3167497,MR3532130,Lerner20}).

For commutators of operators the full picture is still very much incomplete. The first preliminary results for commutators were done by Cruz-Uribe and P\'erez \cite{MR1793688, cruz-uribe-perez02}. The first two authors \cite{MR2918187} showed that if $T$ is a CZO, $b\in BMO$, and
$$\sup_Q\|u^{\frac1p}\|_{L^p(\log L)^{2p-1+\delta},Q}\|v^{-\frac1p}\|_{L^{p'}(\log L)^{2p'-1+\delta},Q}<\infty,$$
then
\begin{equation}\label{commineq}\left(\int_{\R^n}|[b,T]f|^pu\right)^{\frac1p}\leq
  C\|b\|_{BMO}\left(\int_{\R^n}|f|^pv\right)^{\frac1p}.\end{equation}
There is difference in the powers on the logarithmic factors
between the bumps for CZOs versus the bumps for commutators of
CZOs. That is, the powers for CZOs are $p-1+\delta$ and $p'-1+\delta$,
whereas the powers for the commutators are $2p-1+\delta$ and
$2p'-1+\delta$. This additional factor of $2$ in the exponent of the
logarithm is a consequence of the more singular nature of commutators
and is sharp in the sense that one cannot take $\delta=0$ (see
\cite{MR1481632}). In general, commutators behave worse than the
corresponding CZO; for an example of this fact in the one weight
setting,  we refer the reader to \cite{MR2869172}.

A close examination of the proof in \cite{MR2918187} shows that one
can actually separate the bumps in a way (see Remark 2.13 and the
proof Lemma 4.1 in \cite{MR2918187}) to show that if
$\bar{A}\in B_{p'}$ and $\bar{B}\in B_{p}$ and the weights satisfy
\begin{equation}\label{sepcomm}
  \sup_Q\|u^{\frac1p}\|_{L^p(\log
    L)^{2p-1+\delta}}\|v^{-\frac1p}\|_{B,Q}
  +\sup_Q\|u^{\frac1p}\|_{A,Q}\|v^{-\frac1p}\|_{L^p(\log
    L)^{2p-1+\delta}}<\infty,\end{equation}
then inequality \eqref{commineq} holds.

Recently, Lerner, Ombrosi, and Riviera-R\'ios \cite{LORR} have
developed the first techniques for the iterated commutators $T^m_b$.
They showed that if $b\in BMO$, $\bar{A}\in B_{p'}$ and
$\bar{B}\in B_p$, and
\begin{equation}\label{sepcomm}\sup_Q\|u^{\frac1p}\|_{L^p(\log
    L)^{(m+1)p-1+\delta}}\|v^{-\frac1p}\|_{B,Q}
  +\sup_Q\|u^{\frac1p}\|_{A,Q}\|v^{-\frac1p}\|_{L^p(\log
    L)^{(m+1)p-1+\delta}}<\infty, \end{equation}
then
$$\left(\int_{\R^n}|T^m_bf|^pu\right)^{\frac1p}\leq
C\|b\|_{BMO}^m\left(\int_{\R^n}|f|^pv\right)^{\frac1p}.$$
As we were completing this project, we learned that Isralowitz, Pott
and Treil \cite{IsPoTr} have independently been studying the
commutators $[b,T]$ in the two-weight and matrix setting.  They
introduced a class of functions $b$, different from $BMO$, that
interact with the weights $u$ and $v$.  Namely, they are able to show
that the mixed norm condition
\begin{multline*}
  \sup_Q\big\|\|(b(x)-b(y))u^{\frac1p}(x)\|_{A_x,Q}v(y)^{-\frac1p}\big\|_{B_y,Q}\\
  +\sup_Q\big\|\|(b(x)-b(y))v(y)^{-\frac1p}\|_{B_y,Q}u(x)^{\frac1p}\big\|_{A_x,Q}<\infty,
\end{multline*}
where $\bar{A}\in B_{p'}$ and $\bar{B}\in B_p$, is sufficient for the
boundedness $[b,T]:L^p(v)\ra L^p(u)$. Such conditions are related to
our condition below in Theorem \ref{thm:main}; however, our results
hold for general iterated commutators $T^m_b$, when $m>1$.

We remark in passing that there are several interesting results
concerning two weight results for commutators when $u,v\in A_p$ (see
\eqref{eqn:Ap} below for the definition of the $A_p$ class). When the
stronger side assumption that $u,v\in A_p$ is made, it is possible to
provide a complete characterization of the two weight boundedness of
$T^m_b$ in terms of a weighted $BMO$ space.  Such results were
initiated by Bloom \cite{MR805955} for the Hilbert transform and have
since seen a resurgence.  We refer the interested reader to
\cite{MR3451366,MR3606434,MR3919564} and the references therein.

\medskip

We now turn to our results.  We will consider a new bump condition on
weights $u$ and $v$ that directly interact with the
multiplier $b\in L^1_{\mathsf{loc}}(\R^n)$.  Our first main result is
the following.

\begin{theorem} \label{thm:main}
  Suppose $1<p<\infty$, $T$ is a CZO, $m\in \N$ and $b\in
  L^1_{\textsf{loc}}(\R^n)$. If $(u,v)$ is a pair of weights and
  $A,B,C,D$ are Young functions with $\bar{A},\bar{C}\in B_{p'}$,
  $\bar{B},\bar{D}\in B_p$, and  such that
  \begin{multline}\label{twoweightBMO}
    \mathsf{K}=\sup_Q \big\|u^\frac1p\big\|_{A,Q}\big\|(b-b_Q)^mv^{-\frac1p}\big\|_{B,Q}
+\sup_Q \big\|(b-b_Q)^mu^\frac1p\big\|_{C,Q}\big\|v^{-\frac1p}\big\|_{D,Q}<\infty,\end{multline}
then
$$\left(\int_{\R^n}|T^m_b f|^pu\right)^{\frac1p}\lesssim \mathsf{K}\left(\int_{\R^n}|f|^pv\right)^{\frac1p}.$$
\end{theorem} 


As a consequence of this result we can prove more traditional bump
conditions by assuming that the multiplier $b$ lies in an oscillation
class related to $BMO$.  Given a Young function $\Phi$ and the
associated Orlicz space $L^\Phi$, we define the class
$\mathsf{Osc}(\Phi)=\mathsf{Osc}(L^\Phi)$ to be the class of
functions $b\in L^1_{\mathsf{loc}}(\R^n)$ whose semi-norm
$$\|b\|_{\mathsf{Osc}(\Phi)}=\sup_Q\|b-b_Q\|_{\Phi,Q}$$
is finite.  The space $BMO$ corresponds to $\mathsf{Osc}(\Phi)$ for
$\Phi(t)=t$; equivalently $BMO=\mathsf{Osc}(L^1)$.  It is well-known, by
the John-Nirenberg inequality, that $BMO$ functions are exponentially
integrable. Specifically, if $b\in BMO$ then there exist constants
$c,C>0$ such that for any cube $Q$,
\begin{equation}\label{eqn:JN}
  \dashint_Q \exp\left(\frac{c|b(x)-b_Q|}{\|b\|_{BMO}}\right)\,dx\leq C.\end{equation}
This exponential integrability allows us to view $BMO$ as the space of
functions with exponential oscillation.  By the definition of the
Orlicz norm \eqref{eqn:orlicz} and the John-Nirenberg inequality
\eqref{eqn:JN} we have that
$$\|b\|_{\mathsf{Osc}(\Phi)}\lesssim \|b\|_{BMO},$$
where $\Phi(t)=e^t-1$.  Moreover, since $\Phi(t)$ is a Young function
we have $\Phi(t)\gtrsim t$ and hence $\mathsf{Osc}(\Phi)\subseteq BMO$
with
$$\|b\|_{BMO}\lesssim \|b\|_{\mathsf{Osc}(\Phi)}.$$
Thus, that $BMO=\mathsf{Osc}(\Phi)$ where $\Phi(t)=\exp t-1$; for this
Young function it is customary to write $\exp L$ instead of $L^\Phi$,
and we will write $BMO=\mathsf{Osc}(\exp L)$.

The class
$\mathsf{Osc}(\Phi)$ was first introduced in \cite{MR1317714}.  The
importance of condition \eqref{twoweightBMO} is that it allows us to study
commutators for $b\in\mathsf{Osc}(\Phi)$ for a general $\Phi$. This
approach was initiated in \cite{MR4055156}. In particular, we have the
following general result.

\begin{theorem} \label{thm:oscclassbump} Suppose $1<p<\infty$, $T$ is
  a CZO, and $m\in \N$. Further, suppose $A,B,C,D,X,Y, \Phi$ are Young
  functions with $\bar{A},\bar{C}\in B_{p'}$,
  $\bar{B},\bar{D}\in B_{p}$, and $X,Y$ satisfy
  \begin{equation} \label{eqn:checkyoung}
    X^{-1}(t)\lesssim \frac{B^{-1}(t)}{\Phi^{-1}(t)^m}\ \ \text{and} \
    \
    Y^{-1}(t)\lesssim \frac{C^{-1}(t)}{\Phi^{-1}(t)^m}
  \end{equation}
for large $t$.  If $b\in \mathsf{Osc}(\Phi)$ and $(u,v)$ is a pair of weights such that
$$\mathsf{K}=\sup_Q\|u^{\frac1p}\|_{A,Q}\|v^{-\frac{1}{p}}\|_{X,Q}
+\sup_Q\|u^{\frac1p}\|_{Y,Q}\|v^{-\frac{1}{p}}\|_{D,Q}<\infty,$$
 then 
 $$\|T^m_bf\|_{L^p(u)}\lesssim \mathsf{K} \|b\|_{\mathsf{Osc}(\Phi)}^m\|f\|_{L^p(v)}.$$
\end{theorem}


When $b\in BMO$ we may take $\Phi(t)=\exp t-1$ in Theorems
\ref{thm:oscclassbump} and we obtain the following result from
\cite[Theorem 1.1]{LORR}.  

\begin{corollary} \label{thm:BMObump} Suppose that $1<p<\infty$, $T$ is a CZO, $m\in \N$, and $A,D$ are Young functions with $\bar{A}\in B_{p'}$ and $\bar{D}\in B_p$.  If $b\in BMO$ and $(u,v)$ satisfy
 $$\mathsf{K}=\sup_Q\|u^{\frac1p}\|_{A,Q}\|v^{-\frac{1}{p}}\|_{L^{p'}(\log L)^{(m+1)p'-1+\delta},Q}+\sup_Q\|u^{\frac1p}\|_{L^p(\log L)^{(m+1)p-1+\delta},Q}\|v^{-\frac{1}{p}}\|_{D,Q}<\infty$$
for some $\delta>0$, then 
 $$\|T^m_bf\|_{L^p(u)}\lesssim \mathsf{K} \|b\|_{BMO}^m\|f\|_{L^p(v)}.$$
  \end{corollary}



  However, the real power of Theorem \ref{thm:oscclassbump} is that it
  allows us to work with a scale of subclasses of $BMO$ to obtain
  better bump conditions.  Define $\mathsf{Osc}(L^\infty)$ as above
  but using the $L^\infty$ norm, i.e., $\mathsf{Osc}(L^\infty)$ is the space of functions such that $$\sup_Q\|b-b_Q\|_{L^\infty(Q)}<\infty.$$  In fact, we have that
  $\mathsf{Osc}(L^\infty)=L^\infty$.  Clearly, 
  $L^\infty\subseteq \mathsf{Osc}(L^\infty)$.  Conversely, if
  $b\notin L^\infty$, then there exists $Q$ with
  $\|b\|_{\infty,Q}=\infty$.   But  $b_Q$ is a finite constant and so
  $\|b-b_Q\|_{\infty,Q}=\infty$; hence $b\notin \mathsf{Osc}(L^\infty)$.

  Further, we have that
  $$L^\infty=\mathsf{Osc}(L^\infty)\subseteq \mathsf{Osc}(\exp L)=BMO.$$
We will consider Orlicz spaces $L^\Phi$ such that
$$L^\infty \subsetneq \mathsf{Osc}(\Phi)\subsetneq BMO.$$
One such class of functions is $\mathsf{Osc}(\exp L^r)$ for $r>1$,
which corresponds to $\mathsf{Osc}(\Phi)$ where
$\Phi(t)=\exp(t^r)-1$. An alternative definition of
$\mathsf{Osc}(\exp L^r)$ is all  $b\in L^1_{\textsf{loc}}(\R^n)$ that satisfy
$$\dashint_Q \exp(c|b(x)-b_Q|^r)\,dx\leq C$$
for some $c,C>0$ and all cubes $Q$. Placing $b$ in these better
oscillation classes will improve the size of the bump conditions on
the weights.

\begin{corollary} \label{thm:epsbump} Suppose that $1<p<\infty$, $T$ is a CZO, $m\in \N $, and $A,D$ are Young functions with $\bar{A}\in B_{p'}$ and $\bar{D}\in B_{p}$.  If $b\in \mathsf{Osc}(\exp L^\frac1\ep)$ for some $\ep>0$ and $(u,v)$ satisfy
\begin{multline*}\mathsf{K}=\sup_Q\|u^{\frac1p}\|_{A,Q}\|v^{-\frac{1}{p}}\|_{L^{p'}(\log L)^{(\ep m+1)p'-1+\delta},Q}\\+\sup_Q\|u^{\frac1p}\|_{L^p(\log L)^{(\ep m+1)p-1+\delta},Q}\|v^{-\frac{1}{p}}\|_{D,Q}<\infty\end{multline*}
for some $\delta>0$, then 
 $$\|T^m_bf\|_{L^p(u)}\lesssim \mathsf{K} \|b\|_{\mathsf{Osc}(\exp L^\frac1\ep)}^m\|f\|_{L^p(v)}.$$
\end{corollary}

We note in passing that the requirement that $r>1$ is necessary to get
something new.  If $r<1$ and we let $\Phi(t)=\exp(t^r)-1$, then $t
\lesssim \Phi(t) \lesssim e^t-1$, and so $\mathsf{Osc}(\Phi) = BMO$.

\medskip

For $a>0$ the class $\mathsf{Osc}(\exp(L^a))$ is related to 
$$\sqrt[a]{BMO}=\{b\in L^1_{\textsf{loc}}(\R^n): b\geq 0 \ \text{and} \ b^a\in BMO\}.$$
An example of $b\in \sqrt[a]{BMO}$ is given by $b(x)=|\log
x|^{1/a}$. The space $\sqrt[a]{BMO}$ was introduced by Johnson and
Neugebauer in \cite{MR1239426}, but does not seem to have been studied
in depth.  The following theorem is sketched in
\cite{MR1239426}, but  for completeness we give a full proof.

\begin{theorem} \label{thm:bkinBMO} Given $a>1$, $\sqrt[a]{BMO}\subseteq \mathsf{Osc}(\exp(L^a))$, and 
$$\|b\|_{\mathsf{Osc}(\exp(L^a))}\lesssim \|b^a\|_{BMO}^{1/a}.$$
\end{theorem}

The reverse inequality is true for smaller powers: more precisely, if
$0<a\leq 1$ ,then
$$b\in \mathsf{Osc}(\exp(L^a)) \Rightarrow b\in BMO \Rightarrow |b|^a\in BMO.$$
It is not clear whether there is a converse to Theorem \ref{thm:bkinBMO}.


If we assume $b\in \sqrt[a]{BMO}$ for large $a$, for example if
$b\in\bigcap_{k=1}^\infty \sqrt[k]{BMO}$, then we obtain the same
sufficient bump conditions for commutators that hold for CZOs themselves.

\begin{corollary} \label{thm:optimalbump} Suppose that $1<p<\infty$, $m\geq 1$, \ and $A,D$ are Young functions with $\bar{A}\in B_{p'}$ and $\bar{D}\in B_{p}$. Suppose further that $(u,v)$ satisfy
\begin{multline}\label{eqn:seplogbump}\mathsf{K}=\sup_Q\|u^{\frac1p}\|_{A,Q}\|v^{-\frac{1}{p}}\|_{L^{p'}(\log L)^{p'-1+\delta},Q}\\+\sup_Q\|u^{\frac1p}\|_{L^p(\log L)^{p-1+\delta},Q}\|v^{-\frac{1}{p}}\|_{D,Q}<\infty\end{multline}
for some $\delta>0$. If $b\in \sqrt[a]{BMO}$ for $a>\max\{p,p'\}\frac{m}{\delta},$  then 
 $$\|T^m_bf\|_{L^p(u)}\lesssim \mathsf{K}\|b^a\|^{\frac{m}{a}}_{BMO}\|f\|_{L^p(v)}.$$
\end{corollary}

If we take $A(t)=t^p\log(e+t)^{p-1+\delta}$ and
$D(t)=t^{p'}\log(e+t)^{p'-1+\delta}$ (so that $\bar{A}\in B_{p'}$ and
$\bar{B}\in B_p$), then condition \eqref{eqn:seplogbump} becomes
$$\|u^{\frac1p}\|_{L^p(\log L)^{p-1+\delta},Q}\|v^{-\frac{1}{p}}\|_{L^{p'}(\log L)^{p'-1+\delta},Q}<\infty,$$
which is the optimal logarithmic bump condition that is sufficient for
$T$ itself to satisfy $T: L^p(v) \rightarrow L^p(u)$.   This result is
striking given the fact, noted above, that commutators in general are
more singular and so require stronger weight conditions.

We also consider the $b$ in the oscillation class
$\mathsf{Osc}(\exp(\exp(L^r))$ associated with the Young function
$\Phi(t)=\exp(\exp (t^r))-e$.  In particular,
$b\in \mathsf{Osc}(\exp(\exp(L^r))$ if and only if there exist
constants $c,C$ such that
$$\dashint_Q \exp\big(\exp(c|b(x)-b_Q|^r)\big)\,dx\leq C.$$
Since $\exp(t^r)\lesssim \exp(\exp t)$ for large $t$ we have that
$\mathsf{Osc}(\exp(\exp L))\subseteq \bigcap_{r>1} \mathsf{Osc}(\exp
L^r)$ and hence $T^m_b$ will satisfy two weight bounds when
$b\in\mathsf{Osc}\big(\exp(\exp L)\big)$ and the pair of weights
satisfies \eqref{eqn:seplogbump}.  However, we can prove a stronger
condition in the scale of the so-called log-log bumps.

\begin{corollary} \label{cor:doubleexp} Suppose that $1<p<\infty$, $m\geq 1$, and that $A,D$ are Young functions with $\bar{A}\in B_{p'}$ and $\bar{D}\in B_{p}$. Suppose $b\in \mathsf{Osc}(\exp(\exp(L^{\frac1\ep}))$ for some $\ep>0$. If the pair of weights $(u,v)$ satisfy
\begin{multline}\label{eqn:loglogbumpcond}\sup_Q\|u^{\frac1p}\|_{A,Q}\|v^{-\frac1{p}}\|_{L^{p'}(\log L)^{{p'}-1}(\log\log L)^{(1+m\ep)p'-1+\delta}(Q)}\\
+\sup_Q\|u^{\frac1p}\|_{L^p(\log L)^{p-1}(\log\log L)^{(1+m\ep)p-1+\delta}(Q)}\|v^{-\frac1{p}}\|_{D,Q}<\infty,\end{multline}
then 
\begin{equation*}\|T_b^mf\|_{L^p(u)}\leq C\|b\|^m_{\mathsf{Osc}(\exp(\exp L^{\frac1\ep}))}\|f\|_{L^p(v)}.\end{equation*}
\end{corollary}

We note in passing that one can continue to restrict the class of
$BMO$ and get better bump conditions on the weights, passing to
``log-log-log bumps'', etc.  We leave the details to the interested
reader.

\medskip

Finally, we examine some necessary conditions for the iterated
commutators to be bounded.  In Theorem \ref{thm:main} the unbumped
condition for a Calder\'on-Zygmund operator is obtained by taking
$A(t)=C(t)=t^p$ and $B(t)=D(t)=t^{p'}$.  With these Young functions,  condition
\eqref{twoweightBMO} becomes
\begin{equation}\label{unbump}\sup_Q \left(\dashint_Q |b-b_Q|^{mp}u\right)^{\frac1p}\left(\dashint_Q v^{-\frac{p'}{p}}\right)^{\frac{1}{p'}}+\sup_Q \left(\dashint_Q u\right)^{\frac1p}\left(\dashint_Q|b-b_Q|^{mp'} v^{-\frac{p'}{p}}\right)^{\frac{1}{p'}}<\infty.\end{equation}
This condition turns out to be necessary for the sparse operators that
dominate the commutators (see Theorem~\ref{thm:mainsparse} below).
However, for the iterated commutators themselves we can only prove
that a weaker condition
is necessary.  Recall that the Hilbert transform is given by
$$Hf(x)=p.v.\int_{\R}\frac{f(y)}{x-y}\,dy.$$
For first order commutators of the Hilbert transform, or higher order
commutators of even orders, we obtain the following  necessary
conditions.

\begin{theorem}\label{necHilbert} Suppose $1<p<\infty$, $b\in L^1_{\mathsf{loc}}(\R^n)$, $(u,v)$ is a pair of weights, and $H$ is the Hilbert transform on $\R$. If
$$\left(\int_\R|[b,H]f|^pu\right)^{\frac1p}\leq \mathsf{C}\left(\int_\R|f|^pv\right)^{\frac1p},$$
then 
\begin{equation}\label{neccond}\sup_I \left(\frac{1}{v(I)}\int_I|b-b_I|^{p}u\right)^{\frac1p},\sup_I \left(\frac{1}{u^{-\frac{p'}{p}}(I)}\int_I|b-b_I|^{p'}v^{-\frac{p'}{p}}\right)^{\frac1{p'}}\leq \mathsf{C}.\end{equation}
\end{theorem} 

Our techniques extend to even order iterated commutators when $b$ is real.  

\begin{theorem}\label{necHilbertiterated} Suppose $1<p<\infty$, $m\in 2\N$, $b$ is a real valued locally integrable function, $(u,v)$ is a pair of weights, and $H$ is Hilbert transform on $\R$. If
$$\left(\int_\R|H^m_bf|^pu\right)^{\frac1p}\leq \mathsf{C}\left(\int_\R|f|^pv\right)^{\frac1p}$$
then 
\begin{equation}\label{necconditer}\sup_I \left(\frac{1}{v(I)}\int_I|b-b_I|^{mp}u\right)^{\frac1p},\sup_I \left(\frac{1}{u^{-\frac{p'}{p}}(I)}\int_I|b-b_I|^{mp'}v^{-\frac{p'}{p}}\right)^{\frac1{p'}}\leq \mathsf{C}.\end{equation}
\end{theorem} 

The necessity of condition \eqref{neccond} was discovered independently by
Isralowitz, Pott and Treil \cite{IsPoTr} in higher dimensions. We
provide a different proof for the Hilbert transform, which has the
advantage that it yields~\eqref{necconditer} when $m>1$ is even. On
the other hand, while our approach works for the iterated
commutators of even orders, it does not seem to work for odd orders.

It was also observed in
\cite{IsPoTr} that condition \eqref{necconditer} for $m=1$ is
equivalent to the weighted $BMO$ spaces considered in
\cite{MR805955,MR3451366,MR3606434,MR3919564} when $u,v\in A_p$. Thus
Theorem \ref{necHilbert} provides a new proof of the necessity condition
in Bloom's original result.

\medskip

We now consider  the one weight case $u=v$. We say a weight $w\in A_p$ if 
\begin{equation}\label{eqn:Ap}
  [w]_{A_p}=\sup_I\left(\dashint_I w\right)
  \left(\dashint_I w^{-\frac{p'}{p}}\right)^{\frac{p}{p'}}<\infty.\end{equation}
If we combine Theorems \ref{thm:main} and \ref{necHilbert}
in the one weight case and use some well-known properties of $A_p$
weights, then we obtain the following characterization for the Hilbert
transform. When assuming $w\in A_p$ it is well known that $H^m_b$ is
bounded on $L^p(w)$ when $b\in BMO$.  However, the necessity of $BMO$
seems to be new for  higher order commutators (see \cite{MR3716205}
for the necessity when $m=1$).

\begin{corollary} \label{cor:one-wt-H}
Suppose $m\in \{1\}\cup 2\N$, $1<p<\infty$, $b$ is a real valued
function in $L^1_{\mathsf{loc}}(\R)$, $w\in A_p$, and $H$ is the
Hilbert transform on $\R$. Then $H^m_b$ is bounded on $L^p(w)$ if and
only if $b\in BMO$. 
\end{corollary} 

The remainder of the paper is organized as follows. In Section~\ref{section:prelim} we
give the necessary definitions and machinery about approximation of
commutators by so-called sparse operators. We also state two results,
Theorem~\ref{thm:mainsparse} and Corollary~\ref{thm:mainsparse*}, for
the sparse operators that together implies Theorem~\ref{thm:main}. We
prove all our main results in Section~\ref{section:proof}.  
In Section~\ref{section:roots} we prove Theorem
\ref{thm:bkinBMO}
for the class $\sqrt[a]{BMO}$.  Finally, in Section~\ref{section:necc}
we prove the necessary conditions for the
Hilbert transform, Theorems \ref{necHilbert} and
\ref{necHilbertiterated}, and Corollary~\ref{cor:one-wt-H}.

\section{Preliminaries}
\label{section:prelim}
\subsection{Calder\'on-Zygmund operators} We say that $K(x,y)$ defined on $(x,y)\in \R^{2n}$ with $x\not=y$ is a standard kernel if 
$$|K(x,y)|\lesssim \frac{1}{|x-y|^n}, \quad x\not=y,$$
and there exists some $\delta>0$ such that
$$|K(x+h,y)-K(x,y)|+|K(x,y+y)-K(x,y)|\lesssim \frac{|h|^\delta}{|x-y|^{n+\delta}}$$
whenever $|x-y|>2|h|$.  An operator $T$ is said to be a Calder\'on-Zygmund operator if $T$ is bounded on $L^2(\R^n)$ and has the integral representation
$$Tf(x)=\int_{\R^n}K(x,y)f(y)\,dy$$
for all $f\in L^2_c(\R^n)$ and $x\notin \supp f$.

\subsection{Orlicz spaces}

We will need some basic facts and notation for Young functions and
Orlicz spaces.  Here we follow the treatment given
in~\cite[Chapter~5]{MR2797562}.  For functions $A$ and $B$ we will use
the notation $A(t)\lesssim B(t)$ to mean that there exist constants
$c,t_0>0$ such that $A(t)\leq cB(t)$ for $t\geq t_0$ and
$A(t)\eqsim B(t)$ to mean that $A(t)\lesssim B(t)$ and
$B(t)\lesssim A(t)$.  As mentioned in Section 1, every Young function
has an associate Young function. Given $\Phi$ the associate Young
function is defined by
$$\overline{\Phi}(t)=\sup_{s>0}(st-\Phi(s))$$
and satisfies
$$s\leq {\Phi}^{-1}(s)\bar{\Phi}^{-1}(s)\leq 2s.$$
Recall that the Orlicz average is given by
$$\|f\|_{\Phi,Q}=\inf\left\{\lambda>0: \dashint_Q \Phi\Big(\frac{|f(x)|}{\lambda}\Big)\,dx\leq 1\right\}.$$
We will need the general version of H\"older's inequality for Orlicz spaces (see \cite{MR2797562, MR3695871} for details).
\begin{lemma}\label{lem:orliczH} Let $A, B$ be continuous and strictly increasing functions on $[0,\infty)$ and $C$ be Young function that satisfies $A^{-1}(t)B^{-1}(t) \lesssim C^{-1}(t)$ for $t$ large. Then
$$\|fg\|_{C,Q}\lesssim \|f\|_{A}\|g\|_{B}.$$
\end{lemma}
Moreover, we define the Orlicz maximal function
$$M_\Phi f(x)=\sup_{Q\in x} \|f\|_{\Phi,Q}.$$
P\'erez \cite{perez94} showed that $M_\Phi$ is bounded on $L^p$ if and only if $\Phi\in B_p$:
$$\int_1^\infty\frac{\Phi(t)}{t^p}\frac{dt}{t}<\infty.$$
 As mentioned in the introduction, $\Phi(t)=t^{p-\ep}$ and $\Phi(t)=t^p(\log(e+t))^{-1-\delta}$ are typical $B_p$ Young functions. The Young function $\Psi(t)=t^p\log(e+t)^q$ for $p>1$ and $q\in \R$ is called a log-bump and $\Theta(t)=t^p\log(e+t)^q\log[\log(e^e+t)]^r$ for $p>1$ and $q,r\in \R$ is known as a log-log bump.  Will use the following calculations several times in what follows:
 \begin{equation}\label{eqn:loginvers}\Psi^{-1}(t)\eqsim \frac{t^{\frac1p}}{\log(e+t)^{\frac{q}{p}}},\end{equation}
  \begin{equation}\label{eqn:logassoc}\bar{\Psi}(t)\eqsim \frac{t^{p'}}{\log(e+t)^{\frac{p'}{p}q}},\end{equation}
  \begin{equation}\label{eqn:logloginv}\Theta^{-1}(t)\eqsim \frac{t^{\frac1p}}{\log(e+t)^{\frac{q}{p}}\log\log(e^e+t)]^{\frac{r}{p}}},\end{equation}
  \begin{equation}\label{eqn:loglogassoc}\bar{\Theta}(t)\eqsim \frac{t^{p'}}{\log(e+t)^{\frac{p'}{p}q}\log[\log(e^e+t)]^{\frac{p'}{p}r}}.\end{equation}
  These estimates are well-known in the literature on bump conditions (see for example \cite[p. 424]{MR1713140}) and we refer the reader to the book \cite[Chapter 5]{MR2797562} for more details.

\subsection{Sparse Families}

Over the past decade sparse families have played a central role in the study of Calder\'on-Zygumd operators.  For our purpose we will follow the approach from \cite{MR4050113,MR3695871}.  A cube is a subset of $\R^n$ of the form $Q=a+[0,h)^n$ where $a\in \R^n$ and $h>0$. We call $h$ the side length of $Q$ and write $\ell(Q)=h$. A collection of cubes $\mathscr D$ is said to be a dyadic grid if for each $Q\in \mathscr D$, $\ell(Q)=2^k$ for some $k\in \Z$, for each $k\in \Z$ the set $\{Q\in \mathscr D:\ell(Q)=2^k\}$ forms a partition of $\R^n$, and given $Q,P\in \mathscr D$ one has $Q\cap P\in \{\varnothing,P,Q\}$.

Given a dyadic grid $\mathscr D$, a family of cubes
$\mathcal S\subset \mathscr D$ is a {\it sparse family} if there
exists $0<\delta<1$, such that for each $Q\in \Sp$ there exists a
measurable subset $E_Q\subset Q$ with $|E_Q|\geq \delta|Q|$ and the
family $\{E_Q:Q\in \Sp\}$ is disjoint. Recently, Lerner, Ombrosi, and
Rivera-R\'ios \cite{MR3695871} proved a sparse domination formula for
commutators of Calder\'on-Zygmund operators. This was later extended
to the iterated commutators by Iba\~nez-Firnkorn and Rivera-R\'ios
\cite{MR4050113}. In particular they showed that given a CZO $T$,
$b\in L^1_{\textsf{loc}}(\R^n)$, and $f\in L^\infty_c(\R^n)$, there
exist $3^n$ sparse families $\mathcal S_j\subseteq \mathscr D_j$,
$j=1,\ldots, 3^n$, such that
\begin{equation}\label{eqn:sparsedom} |T^m_bf(x)|\lesssim \sum_{j=1}^{3^n}\sum_{Q\in \mathcal S_j}\sum_{k=0}^m|b(x)-b_Q|^{m-k}\left(\dashint_Q|b-b_Q|^kf\right){\mathbf 1}_Q(x).\end{equation}

Our first lemma shows that we can simplify  inequality \eqref{eqn:sparsedom}
by only having to work with the endpoints of the sum corresponding
to $k=0$ and $k=m$.  More precisely, given a sparse family $\mathcal S$ and
$b\in L^1_{\textsf{loc}}(\R^n)$, let
\begin{equation}\label{eqn:sparsecom}T^m_{\mathcal S,b}f(x)=\sum_{Q\in \mathcal S} \left(\dashint_Q |b-b_Q|^mf\right){\mathbf 1}_Q(x).\end{equation}
The adjoint operator $(T^m_{\mathcal S,b})^*$, defined by
$\int (T^m_{\mathcal S,b}f)g=\int f[(T^m_{\mathcal S,b})^*g]$, is
given by
\begin{equation}\label{eqn:sparsecom*}
  (T^m_{\mathcal S,b})^*f(x)=\sum_{Q\in \mathcal S} |b(x)-b_Q|^m
  \left(\dashint_Q f\right){\mathbf 1}_Q(x).\end{equation}

\begin{lemma} \label{lem:sparsebound} Suppose $T$ is a CZO,
  $b\in L^1_{\textsf{loc}}(\R^n)$, and $f\in L^\infty_c(\R^n)$.  Then
  there exist $3^n$ sparse families
  $\mathcal S_j\subseteq \mathscr D_j$, $j=1,\ldots, 3^n$, such that
\begin{equation}\label{eqn:sparsedom2} |T^m_bf(x)|\lesssim \sum_{j=1}^{3^n} T^m_{\mathcal S_j,b}f(x)+(T^m_{\mathcal S_j,b})^*f(x).\end{equation}
\end{lemma}
\begin{proof} Fix a sparse family of cubes $\mathcal S$ and let  $x\in
  \R^n$ and $Q\in \mathcal S$ be such that $x\in Q$.  Then, for each
  $k$, $1<k<m$,
\begin{align*}
  & \sum_{k=0}^m|b(x)-b_Q|^{m-k}\dashint_Q |b(y)-b_Q|^k
    f(y)\,dy \\
&\hspace{3cm}  =\dashint_Q \left(\sum_{k=0}^m|b(x)-b_Q|^{m-k}|b(y)-b_Q|^k\right) f(y)\,dy\\
&\hspace{3cm} \leq \dashint_Q \left(\sum_{k=0}^m\max\{|b(x)-b_Q|,|b(y)-b_Q|\}^{m}\right) f(y)\,dy \\
&\hspace{3cm}= m\dashint_Q \max\{|b(x)-b_Q|^m,|b(y)-b_Q|^m\} f(y)\,dy \\
&\hspace{3cm}\lesssim |b(x)-b_Q|^m\dashint_Q f(y)\,dy+\dashint_Q|b(y)-b_Q|^mf(y) \,dy. 
\end{align*}
Inequality \eqref{eqn:sparsedom2} now follows from \eqref{eqn:sparsedom}.
\end{proof}

We now state our main result for a general sparse operator
$T^m_{\mathcal S,b}$.

\begin{theorem}\label{thm:mainsparse}
Suppose $m\in \N$, $b\in L^1_{\textsf{loc}}(\R^n)$, $\Sp$ is a
  sparse family, and $A,\,B$ are Young functions with
  $\bar{A}\in B_{p'}$ and $\bar{B}\in B_p$.  If $(u,v)$ are a
  pair of weights that satisfy
\begin{equation}\label{eqn:bumpsparse}\sup_{Q\in
    \Sp}\|u^{\frac1p}\|_{A,Q}\|(b-b_Q)^mv^{-\frac{1}{p}}\|_{B,Q}<\infty,
\end{equation}
then the sparse operator $T^m_{\mathcal S,b}$ \eqref{eqn:sparsecom}
satisfies
\begin{equation}\label{eqn:sparseLpineq}
  \|T^m_{\mathcal S,b}f\|_{L^p(u)}\leq C\|f\|_{L^p(v)}\end{equation}
for all $f\in L^p(v)$.

Conversely, if $T^m_{\mathcal S,b}$ satisfies \eqref{eqn:sparseLpineq},
then the pair of weights $(u,v)$ satisfies \eqref{eqn:bumpsparse} with
$A(t)=t^p$ and $B(t)=t^{p'}$: that is,
$$\sup_{Q\in \Sp} \left(\dashint_Qu\right)^{\frac1p} \left(\dashint_Q|b-b_Q|^{mp'}v^{-\frac{p'}{p}}\right)^{\frac1{p'}}<\infty.$$
\end{theorem}

As a corollary we obtain the corresponding result for the adjoint
operator $(T^m_{\mathcal S,b})^*$, since any linear operator $S$
satisfies $S:L^p(v)\ra L^p(u)$ if and only if
$S^*:L^{p'}(u^{-\frac{p'}{p}})\ra L^{p'}(v^{-\frac{p'}{p}})$ by duality.

\begin{corollary} \label{thm:mainsparse*} Suppose $m\in \N$,
  $b\in L^1_{\textsf{loc}}(\R^n)$, $\Sp$ is a sparse family, and
  $C,D$ are Young functions with $\bar{C}\in B_{p'}$ and
  $\bar{D}\in B_p$.  If $(u,v)$ are a pair of weights that
  satisfy
$$\sup_{Q\in
  \Sp}\|(b-b_Q)^mu^{\frac1p}\|_{C,Q}\|v^{-\frac{1}{p}}\|_{D,Q}<\infty,$$
then the sparse operator $(T^m_{\mathcal S,b})^*$
\eqref{eqn:sparsecom*} satisfies
\begin{equation}\label{eqn:sparseLpineq*}
  \|(T^m_{\mathcal S,b})^*f\|_{L^p(u)}\leq C\|f\|_{L^p(v)}\end{equation}
for all $f\in L^p(v)$.

Conversely, if $(T^m_{\mathcal S,b})^*$ satisfies
\eqref{eqn:sparseLpineq*}, then the pair of weights $(u,v)$ satisfies
$$\sup_{Q\in \Sp} \left(\dashint_Q|b-b_Q|^{mp}u\right)^{\frac1p} \left(\dashint_Qv^{-\frac{p'}{p}}\right)^{\frac1{p'}}\leq C.$$
\end{corollary}

\section{Proofs}
\label{section:proof}

In this section we prove our main results.  Theorem \ref{thm:main}
follows immediately from Lemma \ref{lem:sparsebound}, Theorem
\ref{thm:mainsparse}, and Corollary \ref{thm:mainsparse*}.  We now
prove Theorem \ref{thm:mainsparse}.

\begin{proof}[Proof of Theorem \ref{thm:mainsparse}] Let $\Sp$ be a
  sparse family and let $T^m_{\mathcal S,b}$ be the associated
  operator~\eqref{eqn:sparsecom}. 
Furthermore, let 
$$\mathsf{K}=\sup_{Q\in \Sp} \|u^{\frac1p}\|_{A,Q}\|(b-b_Q)^mv^{-\frac{1}{p}}\|_{B,Q}.$$
To estimate $\|T^m_{\mathcal S,b}f\|_{L^p(u)}$ we bound the bilinear
form $\int (T^m_{\mathcal S,b}f)gu$ for $g\in L^{p'}(u)$.  We have
\begin{align*}\int_{\R^n}(T^m_{\mathcal S,b}f)gu\,dx&=\sum_{Q\in \Sp}\left(\dashint_Q|b-b_Q|^mf \right)\left(\dashint_Q gu\right) |Q| \\ 
&\leq \sum_{Q\in \Sp} \|(b-b_Q)^mv^{-\frac{1}{p}}\|_{B,Q} \|fv^{\frac{1}{p}}\|_{\bar{B},Q} \|u^{\frac1p}\|_{A,Q}\|gu^{\frac{1}{p'}}\|_{\bar{A},Q}|Q|\\
 &\lesssim \mathsf{K} \sum_{Q\in \Sp}
 \|fv^{\frac{1}{p}}\|_{\bar{B},Q}\|gu^{\frac{1}{p'}}\|_{\bar{A},Q}|E_Q|\\
&\leq \mathsf{K}
 \int_{\R^n}M_{\bar{B}}(fv^{\frac{1}{p}})M_{\bar{A}}(gu^{\frac{1}{p'}})\,dx\\
  &\leq
    \mathsf{K}\|M_{\bar{B}}(fv^{\frac{1}{p}})\|_{L^p}\|M_{\bar{A}}(gu^{\frac{1}{p'}})\|_{L^{p'}} \\
&\leq \mathsf{K} \|M_{\bar{B}}\|_{\mathcal B(L^p)} \|M_{\bar{A}} \|_{\mathcal B(L^{p'})} \|f\|_{L^p(v)}\|g\|_{L^{p'}(u)}.
\end{align*}

To prove necessity, let $\sigma=v^{-\frac{p'}{p}}$, fix $Q\in \Sp$ 
such that $\int_Q|b-b_Q|^{mp'}\sigma \,dx>0$ (since any other cube
will not contribute to the supremum), and define $f$ by
$$f=|b-b_Q|^{m(p'-1)}\sigma{\mathbf 1}_Q.$$  Then  $f\geq 0$ and for $x\in Q$,
$$T^m_{\mathcal S,b}f(x)\geq \dashint_Q |b-b_Q|^mf\,dx=\dashint_Q |b-b_Q|^{mp'}\sigma\,dx.$$
If we plug this into the norm inequality, we have
\begin{multline*}
 \dashint_Q |b-b_Q|^{mp'}\sigma \; \left(\int_Q
   u\right)^{\frac1p}
 \leq \left(\int_{Q} (T^m_{\mathcal S,b}f)^p u\right)^{\frac1p}\\
 \leq \left(\int_{\R^n} (T^m_{\mathcal S,b}f)^p u\right)^{\frac1p} 
\leq C \left(\int_{\R^n}f^pv\right)^{\frac1p} 
= C\left(\int_{Q} |b-b_Q|^{mp'} \sigma \right)^{\frac1p}.
\end{multline*}
If we rearrange terms, we get
$$\left(\dashint_Q u\,dx\right)^{\frac1p} \left(\dashint_Q |b-b_Q|^{mp'}\sigma\,dx\right)^{\frac{1}{p'}}\leq C.$$ 
\end{proof}

\begin{proof}[Proof of Theorem \ref{thm:oscclassbump}]
  This result is a corollary of Theorem \ref{thm:main} and the Orlicz
  H\"older inequality, Lemma~\ref{lem:orliczH}.  Indeed, since $B,X,\Phi$ satisfy
$$\Phi^{-1}(t)^mX^{-1}(t)\lesssim B^{-1}(t)$$
for $t$ large, if we let $\Phi_m(t)=\Phi(t^{\frac1m})$, we have that
$$\|(b-b_Q)^mv^{-\frac1p}\|_{B,Q}\lesssim \|(b-b_Q)^m\|_{\Phi_m,Q}\|v^{-\frac1p}\|_{X,Q}\eqsim\|(b-b_Q)\|_{\Phi,Q}^m\|v^{-\frac1p}\|_{X,Q}.$$
Here we  used that
$$\|f^m\|_{\Phi_m,Q}=\|f\|_{\Phi,Q}^m,$$
which holds for any Young function when $\Phi_m(t)=\Phi(t^{\frac1m})$.
Hence,
$$\sup_Q\|u^{\frac1p}\|_{A,Q}\|(b-b_Q)^mv^{-\frac1p}\|_{B,Q}\lesssim \|b\|_{\mathsf{Osc}(\Phi)}^m\sup_Q\|u^{\frac1p}\|_{A,Q}\|v^{-\frac1p}\|_{X,Q}<\infty.$$
A similar argument shows that
$$\sup_Q\|(b-b_Q)^mu^{\frac1p}\|_{C,Q}\|v^{-\frac1p}\|_{D,Q}\lesssim  \|b\|_{\mathsf{Osc}(\Phi)}^m\sup_Q\|u^{\frac1p}\|_{Y,Q}\|v^{-\frac1p}\|_{D,Q}<\infty,$$
and thus the hypotheses of Theorem \ref{thm:main} are satisfied.
\end{proof}

\begin{proof}[Proof of Corollary \ref{thm:epsbump}] Define 
\begin{align*}B(t)&=t^{p'}\log(e+t)^{p'-1+\delta},\\
C(t)&=t^{p}\log(e+t)^{p-1+\delta},
\end{align*}
for some $\delta>0$. Then by \eqref{eqn:logassoc} we have that
$$\bar{B}(t)\eqsim\frac{t^{p}}{\log(e+t)^{1+\frac{p}{p'}\delta}}
\ \ \text{and} \ \ \bar{C}(t)\eqsim\frac{t^{p'}}{\log(e+t)^{1+\frac{p'}{p}\delta}},$$
so that $\bar{B}\in B_{p}$ and $\bar{C}\in B_{p'}$.
Now define $X,Y,$ and $\Phi$ by 
\begin{align*}
X(t)&=t^{p'}\log(e+t)^{(1+m\ep)p'-1+\delta},\\
Y(t)&=t^{p}\log(e+t)^{(1+m\ep)p-1+\delta},\\
\Phi(t)&=\exp(t^{\frac1\ep})-1.
\end{align*} 
Then  $\mathsf{Osc}(\Phi)=\mathsf{Osc}(\exp L^{\frac1\ep})$.  We will
show that the two conditions in~\eqref{eqn:checkyoung} in
Theorem~\ref{thm:oscclassbump} hold: that is, that 
$$\Phi^{-1}(t)^mX^{-1}(t)\lesssim B^{-1}(t) \ \ \text{and}\ \  \Phi^{-1}(t)^mY^{-1}(t)\lesssim C^{-1}(t)$$
hold for large $t$.  We will prove the required inequality for the
triple $B,X,$ and $\Phi$; the proof of the other inequality for $C,Y,$
and $\Phi$ is the same.
By~\eqref{eqn:loginvers} we have that
\begin{align*}
B^{-1}(t)&\eqsim\frac{t^{\frac1{p'}}}{\log(e+t)^{\frac1{p}+\frac{\delta}{p'}}},\\
\Phi^{-1}(t)&\eqsim\log(e+t)^\ep, \\
X^{-1}(t)&\eqsim\frac{t^{\frac{1}{p'}}}{\log(e+t)^{m\ep+\frac{1}{p}+\frac{\delta}{p'}}},
\end{align*}
and hence,
$$\Phi^{-1}(t)^mX^{-1}(t)\eqsim\log(e+t)^{m\ep}\frac{t^{\frac1{p'}}}{\log(e+t)^{m\ep+\frac1{p}+\frac{\delta}{p'}}}=\frac{t^{\frac1{p'}}}{\log(e+t)^{\frac1{p}+\frac{\delta}{p'}}}\eqsim B^{-1}(t).$$
\end{proof}

We now prove Corollary \ref{thm:optimalbump}.  This
result will follow from the fact that
$\sqrt[a]{BMO}\subseteq \mathsf{Osc}(\exp L^a)$ (Theorem
\ref{thm:bkinBMO}) and the following  result, which roughly says that
we may take $\ep=0$ in Corollary~\ref{thm:epsbump}.

\begin{theorem} \label{thm:optimalbumpexpclass} Suppose that
  $1<p<\infty$, $m\geq 1$, \ and $A,D$ are Young functions with
  $\bar{A}\in B_{p'}$ and $\bar{D}\in B_{p}$. Suppose further that
  the pair $(u,v)$ satisfies
  \begin{equation*}
    \mathsf{K}=\sup_Q\|u^{\frac1p}\|_{A,Q}\|
    v^{-\frac{1}{p}}\|_{L^{p'}(\log L)^{p'-1+\delta},Q}
    +\sup_Q\|u^{\frac1p}\|_{L^p(\log L)^{p-1+\delta},Q}
    \|v^{-\frac{1}{p}}\|_{D,Q}<\infty
  \end{equation*}
  for some $\delta>0$. If $b\in \mathsf{Osc}(\exp L^{\frac1\ep})$ for
  $0<\ep<\frac{\delta}{m\max\{p,p'\}}$, then 
 $$\|T^m_bf\|_{L^p(u)}\lesssim \mathsf{K}\|b\|^{m}_{\mathsf{Osc}(\exp L^{\frac{1}{\ep}})}\|f\|_{L^p(v)}.$$
\end{theorem}

\begin{proof}[Proof of Theorem \ref{thm:optimalbumpexpclass}]
  Again we will use Theorem \ref{thm:oscclassbump}. Let $\delta>0$ be
  as in the statement of theorem. Define
$$\alpha=\delta-\ep m p' \quad \text{and} \quad \beta=\delta-\ep m p,$$
so that $\al,\beta>0$.   Now define
\begin{align*}B(t)&=t^{p'}\log(e+t)^{{p'}-1+\al}\\
C(t)&=t^{p}\log(e+t)^{p-1+\beta}
\end{align*}
so that $$\bar{B}(t)\eqsim\frac{t^{p}}{\log(e+t)^{1+\frac{p}{p'}\al}} \ \ \text{and} \ \ \bar{C}(t)\eqsim\frac{t^{p}}{\log(e+t)^{1+\frac{p'}{p}\beta}},$$
which satisfy $\bar{B}\in B_{p}$ and $\bar{C}\in B_{p'}$.
Further, define $X,Y,$ and $\Phi$ by 
\begin{align*}
X(t)&=t^{p'}\log(e+t)^{p'-1+\delta},\\
Y(t)&=t^{p}\log(e+t)^{p-1+\delta},\\
\Phi(t)&=\exp(t^{\frac1\ep})-1.
\end{align*} 
Then we have that
$$\Phi^{-1}(t)^mX^{-1}(t)\eqsim
\log(e+t)^{m\ep}\frac{t^{\frac{1}{p'}}}{\log(e+t)^{\frac1p+\frac{\delta}{p'}}}
=\frac{t^{\frac{1}{p'}}}{\log(e+t)^{\frac1p+\frac{\al}{p'}}}\eqsim B^{-1}(t);$$
similarly,
$$\Phi^{-1}(t)^mY^{-1}(t)\eqsim C^{-1}(t),$$ 
and so  the hypotheses of Theorem \ref{thm:oscclassbump} are satisfied.  
\end{proof}

The proof of Corollary \ref{cor:doubleexp} follows a similar argument
to that of Corollary \ref{thm:epsbump} and we will only sketch it here.
Define the Young functions 
\begin{align*}B(t)&=t^{p'}\log(e+t)^{p'-1}\log[\log(e^e+t)]^{p'-1+\delta},\\
C(t)&=t^{p}\log(e+t)^{p-1}\log[\log(e^e+t)]^{p-1+\delta},\\
X(t)&=t^{p'}\log(e+t)^{p'-1}\log[\log(e+t)]^{(1+m\ep)p'-1+\delta},\\
Y(t)&=t^{p}\log(e+t)^{p-1}\log[\log(e+t)]^{(1+m\ep)p-1+\delta},\\
\Phi(t)&=\exp[\exp(t^{\frac1\ep})]-e.
\end{align*}
Then \eqref{eqn:logloginv} and \eqref{eqn:loglogassoc} imply
$$B^{-1}(t)\eqsim
\frac{t^{\frac{1}{p'}}}{\log(e+t)^{\frac1p}\log[\log(e^e+t)]^{\frac1p+\frac{\delta}{p'}}}, $$
$$X^{-1}(t)\eqsim
\frac{t^{\frac{1}{p'}}}{\log(e+t)^{\frac1p}\log[\log(e^e+t)]^{m\ep+\frac1p+\frac{\delta}{p'}}}; $$
$$\Phi^{-1}(t)\eqsim\log[\log(e+t)]^\ep.$$
hence,
$$\Phi^{-1}(t)^mX^{-1}(t)\eqsim \log[\log(e^e+t)]^{m\ep}\frac{t^{\frac{1}{p'}}}{\log(e+t)^{\frac1p}\log[\log(e^e+t)]^{m\ep+\frac1p+\frac{\delta}{p'}}}\eqsim B^{-1}(t).$$

\section{Roots of $BMO$ functions}
\label{section:roots}

In this section we prove Theorem \ref{thm:bkinBMO}.  Recall that that
for $a>1$ we define the space 
$$\sqrt[a]{BMO}=\{b\in L^1_{\textsf{loc}}(\R^n): b\geq 0 \ \text{and}
\ b^a\in BMO\}.$$
Some of what we do is implicit in~\cite{MR1239426}, however, here we
give a more complete and systematic treatment.  We will make extensive use of
the function $F(x)=x^{1/a}$, defined on $x\geq 0$, which satisfies $F(x^a)=x$.  Since
$a>1$, $F$ is H\"older continuous of order $1/a$, that is,
\begin{equation} \label{eqn:holder}|F(x)-F(y)|\leq |x-y|^{\frac1a}.\end{equation}

We first  observe that $\sqrt[a]{BMO}\subseteq BMO$, when $a>1$.  
Indeed, given a cube $Q$ and $b\in \sqrt[a]{BMO}$ let $c_Q=F\big((b^a)_Q\big)$.  Then
\begin{multline*}
  \dashint_Q |b(x)-c_Q|\,dx=\dashint_Q
  |F(b(x)^a)-F\big((b^a)_Q\big)|\,dx\\
  \leq \dashint_Q |b(x)^a-(b^a)_Q)|^{\frac1a}\,dx\leq
  \|b^a\|_{BMO}^{\frac1a}.
\end{multline*} 

\begin{proof}[Proof of Theorem \ref{thm:bkinBMO}]
  Let $b\in \sqrt[a]{BMO}$;  without loss of generality we may assume
  that $\|b^a\|_{BMO}=1$.  The general case when $\|b^a\|_{BMO}\not=0$
  now follows by homogeneity if we replace $b$ by $b/\|b^a\|^{\frac1a}_{BMO}$.
  Let $u=b^a$ so that $u\in BMO$.   By \eqref{eqn:JN} there exist constants $c,C>0$ such that
$$\dashint_Q \exp(c|u-u_Q|)\leq C$$
for all cubes $Q$.  Let $F(x)=x^{\frac1a}$, $x\geq 0$; then $F(u)=b$.  Now fix a cube $Q$ and let $A=\frac{c}{2^a}$.  Then
\begin{align*}
\int_Q\exp(A|b-b_Q|^a)&= \int_Q\exp(A|b-F(u_Q)+F(u_Q)-b_Q|^a)\\
&\leq \int_Q\exp(c|b-F(u_Q)|^a+c|F(u_Q)-b_Q|^a)\\
&=  \exp(c|F(u_Q)-b_Q|^a)\int_Q\exp(c|F(u)-F(u_Q)|^a)\\
&\leq  \exp(c|F(u_Q)-b_Q|^a)\int_Q\exp(c|u-u_Q|)\\
&\leq  C\exp(c|F(u_Q)-b_Q|^a))|Q|,
\end{align*}
where we used inequality \eqref{eqn:holder} in the second to last
inequality. Then, since $a>1$,
\begin{multline*}
  |F(u_Q)-b_Q|^a
  =\left|\dashint_Q (F(u_Q)-b(x))\,dx\right|^a
  \leq \dashint_Q|b(x)-F(u_Q)|^a\,dx\\
  = \dashint_Q |F(u)-F(u_Q)|^a\leq \dashint_Q |u-u_Q|,
\end{multline*}
where we again  used inequality \eqref{eqn:holder}.  Hence, by
Jensen's inequality,  
$$\exp(c|F(u_Q)-b_Q|^a))\
\leq \exp\left(c\dashint_Q|u-u_Q|\right)
\leq \dashint_Q\exp(c|u-u_Q|)\leq C.$$
Therefore,
$$\int_Q\exp(A|b-b_Q|^a)\leq C^2|Q|,$$
which in turn implies that
$$\|b-b_Q\|_{\exp(L^a),Q}\leq K,$$
where $K$ depends on $A$ and $C$. Thus $b\in \mathsf{Osc}(\exp L^a)$. 
\end{proof}

\section{Necessary conditions for the Hilbert transform}
\label{section:necc}

In this section we prove the necessity conditions for the
commutators of the Hilbert transform.  We will follow the approach of
Hyt\"onen \cite{Hy1} for the Beurling transform (see also
\cite{HyLiOi}), although we note that these proofs are for
\emph{unweighted} estimates.

\begin{proof}[Proof of Theorems \ref{necHilbert} and \ref{necHilbertiterated}]
We first prove the case when $m=1$.  Suppose $(u,v)$ are weights and $b\in L^1_{\textsf{loc}}(\R)$ such that $[b,H]:L^p(v)\ra L^p(u)$ with
$$\|[b,H]f\|_{L^p(u)}\leq \mathsf{C}\|f\|_{L^p(v)}.$$  
Let $I=[a,b]$ be a fixed interval and $c=(a+b)/2$ be its center.
Define $$f(x)=\sgn(b(x)-b_I)|b(x)-b_I|^{p-1}{\mathbf 1}_I(x);$$
then
\begin{align*}
  & \int_I |b(x)-b_I|^pu(x)\,dx \\
  & \qquad \qquad=\int_I (b(x)-b_I)f(x)u(x)\,dx\\
& \qquad \qquad=\int_I\dashint_I (b(x)-b(y))f(x)u(x)\,dydx\\
& \qquad \qquad=\int_I\dashint_I \frac{(b(x)-b(y))}{x-y}(x-y)f(x)u(x)\,dydx\\
& \qquad \qquad=\int_I\dashint_I \frac{(b(x)-b(y))}{x-y}(x-c+c-y)f(x)u(x)\,dydx\\
& \qquad \qquad=\int_\R\left(\int_\R \frac{(b(x)-b(y))}{x-y}{\mathbf 1}_I(y)\,dy\right)\frac{x-c}{|I|}{\mathbf 1}_I(x)f(x)u(x)\,dx\\
& \qquad \qquad\quad +\int_I\left(\int_\R \frac{(b(x)-b(y))}{x-y}\frac{c-y}{|I|}{\mathbf 1}_I(y)\,dy\right){\mathbf 1}_I(x)f(x)u(x)\,dydx\\
& \qquad \qquad=\int_\R [b,H]({\mathbf 1}_I)(x)g_c(x)f(x)u(x)\,dx-\int_\R [b,H](g_c)(x)f(x)u(x)\,dx,
\end{align*}
where 
$$g_c(u)=\frac{u-c}{|I|}\mathbf 1_I(u).$$
Note that $\text{supp}(g_c)\subseteq I$ and $g\in L^\infty$ with
$\|g_c\|_\infty\leq \frac12.$ Hence,
\begin{align*}\left|\int_\R [b,H]({\mathbf 1}_I)(x)g_c(x)f(x)u(x)\,dx\right|&\leq \int_\R |[b,H]({\mathbf 1}_I)(x)g_c(x)||f(x)|u(x)\,dx\\
& \leq \|[b,H]({\mathbf 1}_I)\|_{L^p(u)}\|g_cf\|_{L^{p'}(u)}\\
&\leq \frac{\mathsf{C}}{2}\|{\mathbf 1}_I\|_{L^p(v)}\|f\|_{L^{p'}(u)}\\
&\leq \frac{\mathsf{C}}{2}v(I)^{\frac1p}\left(\int_I|b-b_I|^pu\right)^{\frac{1}{p'}}.
\end{align*}
Similarly, the second term satisfies
\begin{multline*}
  \left|\int_\R [b,H](g_c)(x)f(x)u(x)\,dx\right|
  \leq \|[b,H](g_c)\|_{L^p(u)}\|f\|_{L^{p'}(u)}\\
  \leq \mathsf{C} \|g_c\|_{L^p(v)} \|f\|_{L^{p'}(u)}
=\frac{\mathsf{C}}{2}v(I)^{\frac1p}\left(\int_I|b-b_I|^pu\right)^{\frac{1}{p'}}.
\end{multline*}
Thus, we have
\begin{multline*}\int_I |b-b_I|^pu \leq \left|\int_\R [b,H]({\mathbf 1}_I)(x)g_c(x)f(x)u(x)\,dx\right|+\left|\int_\R [b,H](g_c)(x)f(x)u(x)\,dx\right|\\ \leq \mathsf{C}v(I)^{\frac1p}\left(\int_I|b-b_I|^pu\right)^{\frac{1}{p'}}.\end{multline*}
It now follows that
 $$\left(\frac{1}{v(I)}\int_I |b-b_I|^pu\right)^{\frac1p}\leq \mathsf{C}.$$

 The other estimate,
$$\left(\frac{1}{u^{-\frac{p'}{p}}(I)}\int_I
  |b-b_I|^{p'}v^{-\frac{p'}{p}}\right)^{\frac1{p'}}\leq \mathsf{C},$$
follows from duality by interchanging the roles of the weights $(u,v)$
with $(v^{-\frac{p'}{p}},u^{-\frac{p'}{p}})$.

\medskip

For the case $m=2k$ we will assume that $b$ is real valued.  The even iterated commutators,
$$H^{2k}_bf(x)=p.v.\int_\R\frac{(b(x)-b(y))^{2k}}{x-y}f(y)\,dy,$$
have a positivity that we will exploit.  Let 
$$f(x)=|b(x)-b_I|^{2k(p-1)}{\mathbf 1}_I(x),$$
and notice that $f\geq 0$. Then
\begin{multline*}\int_I |b(x)-b_I|^{2kp}u(x)\,dx=\int_I |b(x)-b_I|^{2k}f(x)u(x)\,dx\\
=\int_I \left|\dashint_I(b(x)-b(y))\,dy\right|^{2k}f(x)u(x)\,dx\leq \int_I \dashint_I(b(x)-b(y))^{2k}\,f(x)u(x)\,dydx.\end{multline*}
It is exactly at this point that we have used that $m=2k$ is even and
$b$ is real, since $$|b(x)-b(y)|^{2k}=(b(x)-b(y))^{2k}.$$
Then
\begin{align*}
  & \int_I |b(x)-b_I|^{2kp}u(x)\,dx \\
  & \qquad \qquad \leq  \int_I \dashint_I(b(x)-b(y))^{2k}\,f(x)u(x)\,dydx\\
  &\qquad \qquad
    =\int_\R\left(\int_\R \frac{(b(x)-b(y))^{2k}}{x-y}{\mathbf 1}_I(y)\,dy\right)\frac{x-c}{|I|}{\mathbf 1}_I(x)f(x)u(x)\,dx\\
  &\qquad \qquad \quad
    +\int_I\left(\int_\R \frac{(b(x)-b(y))^{2k}}{x-y}
    \frac{c-y}{|I|}{\mathbf 1}_I(y)\,dy\right)\chi_I(x)f(x)u(x)\,dydx\\
  &\qquad \qquad =\int_\R H^{2k}_b({\mathbf 1}_I)(x)g_c(x)f(x)u(x)\,dx
    -\int_\R H^{2k}_b(g_c)(x)f(x)u(x)\,dx.
\end{align*}
By H\"older's inequality and the boundedness of $H^{2k}_b:L^p(v)\ra
L^p(u)$.  we get that
\begin{multline*} \int_I |b-b_I|^{2kp}u\leq \|H^{2k}_b({\mathbf 1}_I)\|_{L^p(u)}\|g_cf\|_{L^{p'}(u)}+\|H^{2k}_b(g_c)\|_{L^p(u)}\|f\|_{L^{p'}(u)}\\
\leq \mathsf{C}\|{\mathbf 1}_I\|_{L^p(v)}\|g_c\|_\infty\|f\|_{L^{p'}(u)}+\mathsf{C}\|g_c\|_{L^p(v)}\|f\|_{L^{p'}(u)}\leq \mathsf{C}v(I)^{\frac1p}\left(\int_I |b-b_I|^{2kp}u\right)^{\frac1{p'}}.\end{multline*}
The rest of the proof follows as in the case $m=1$.
\end{proof}

\medskip

Finally, we prove Corollary~\ref{cor:one-wt-H}.  First note that if
$w\in A_p$, then given any cube $Q$
$$\left(\frac{|Q|}{w(Q)}\right)^{\frac1p}\leq \left(\dashint_Q w^{-\frac{p'}{p}}\right)^{\frac1{p'}}\leq [w]^{\frac1p}_{A_p}\left(\frac{|Q|}{w(Q)}\right)^{\frac1p}.$$
Therefore, we have
\begin{multline*}
  \left(\frac{1}{w(Q)}\int_I |b-b_Q|^{mp}w\right)^{\frac1p} \\
  \leq \left(\dashint_Q |b-b_Q|^{mp}w\right)^{\frac1{p}}
  \left(\dashint_Q w^{-\frac{p'}{p}}\right)^{\frac1{p'}}
\leq [w]^{\frac1p}_{A_p}\left(\frac{1}{w(Q)}\int_Q
  |b-b_Q|^{mp}w\right)^{\frac1p};
\end{multline*}
 similarly,
 \begin{multline*}
   \left(\frac{1}{w^{-\frac{p'}{p}}(Q)}\int_Q
     |b-b_Q|^{mp}w^{-\frac{p'}{p}}\right)^{\frac1{p'}} \\
   \leq \left(\dashint_Q w\right)^{\frac1{p}}
   \left(\dashint_Q|b-b_Q|^{mp}w^{-\frac{p'}{p}}\right)^{\frac1{p'}}
   \leq [w]^{\frac1p}_{A_p}\left(\frac{1}{w^{-\frac{p'}{p}}(Q)}
     \int_Q |b-b_Q|^{mp}w^{-\frac{p'}{p}}\right)^{\frac1{p'}}.
 \end{multline*}

 \begin{proof}[Proof of Corollary~\ref{cor:one-wt-H}]
   First suppose that $H^m_b$ is bounded on $L^p(w)$ for $m\in \{1\}\cup 2\N$, with operator norm $\|H_b^mf\|_{L^p(w)}\leq \mathsf{C}\|f\|_{L^p(w)}$.  If $I$ is a fixed interval, then by H\"older's inequality and \eqref{neccond} we have
   \begin{multline*}\label{eqn:bmo}
     \dashint_I |b-b_I|^m\leq
     \left(\dashint_I|b-b_I|^{mp}w\right)^{\frac1p}
     \left(\dashint_I w^{-\frac{p'}{p}}\right)^{\frac1{p'}} \\ 
\leq [w]_{A_p}^{\frac1p}
\left(\frac{1}{w(I)}\int_I|b-b_I|^{mp}w\right)^{\frac1p}
\leq \mathsf{C}[w]_{A_p}^{\frac1p}.
\end{multline*}
Since $I$ is arbitrary, we get that $b\in BMO$ with
$$\|b\|_{BMO}^m\leq \mathsf{C}[w]_{A_p}^{\frac1p}.$$

On the other hand, suppose $b\in BMO$.  We will find Young functions
$A,B,C,D$ such that the weight $w$ and $b$ satisfy the hypothesis of
Theorem~\ref{thm:main},  thus showing that $H^m_b$ is bounded on
$L^p(w)$. Since $w\in A_p$, there exists $r>1$,  a reverse H\"older
exponent, such that there exists a constant $C$ with
$$\left(\dashint_I w^r\right)^{\frac1r}\leq C\dashint_I w, \ \ \text{and} \ \ \left(\dashint_I w^{-r\frac{p'}{p}}\right)^{\frac1r}\leq C\dashint_I w^{-\frac{p'}{p}}$$
for any interval $I$. If $1<s<r$, then
\begin{multline*}
  \left(\dashint_I |b-b_I|^{msp}w^s\right)^{\frac{1}{sp}}
  \left(\dashint_I w^{-s\frac{p'}{p}}\right)^{\frac{1}{sp'}}
  \leq \left(\dashint_I
    |b-b_I|^{m\frac{srp}{r-s}}\right)^{\frac{r-s}{srp}}
  \left(\dashint_Iw^r\right)^{\frac{1}{rp}}
  \left(\dashint_I w^{-r\frac{p'}{p}}\right)^{\frac{1}{rp'}}\\
  \lesssim \|b\|^m_{BMO}[w]^{\frac1p}_{A_p};
\end{multline*}
 similarly,
 $$\left(\dashint_I w^{s}\right)^{\frac{1}{sp}}
 \left(\dashint_I
   |b-b_I|^{msp'}w^{-s\frac{p'}{p}}\right)^{\frac{1}{sp'}}
 \lesssim \|b\|^m_{BMO}[w]^{\frac1p}_{A_p}.$$
In particular,  condition \eqref{twoweightBMO},
$$\sup_Q
\big\|(b-b_Q)^{m}w^\frac1p\big\|_{A,Q}\big\|w^{-\frac1p}\big\|_{B,Q}
+\sup_Q
\big\|w^\frac1p\big\|_{C,Q}\big\|(b-b_Q)^mw^{-\frac1p}\big\|_{D,Q}
<\infty,$$ is satisfied with $A(t)=C(t)=t^{sp}$ and
$B(t)=D(t)=t^{sp'}$. Moreover, we have $\bar{A},\bar{C}\in B_{p'}$ and
$\bar{B},\bar{D}\in B_p$ since $s>1$. Theorem \ref{thm:main} now
implies that if $b\in BMO$, then $H^m_b:L^p(w)\ra L^p(w)$.
\end{proof}


\bibliographystyle{plain}
\bibliography{Newcommutator}

\begin{thebibliography}{10}

\bibitem{MR3302574}
T.C. Anderson, D.~Cruz-Uribe, and K.~Moen.
\newblock Logarithmic bump conditions for {C}alder\'{o}n-{Z}ygmund operators on
  spaces of homogeneous type.
\newblock {\em Publ. Mat.}, 59(1):17--43, 2015.

\bibitem{MR4055156}
\'{A}. B\'{e}nyi, J.M. Martell, K.~Moen, E.~Stachura, and R.H. Torres.
\newblock Boundedness results for commutators with {BMO} functions via weighted
  estimates: a comprehensive approach.
\newblock {\em Math. Ann.}, 376(1-2):61--102, 2020.

\bibitem{MR805955}
S.~Bloom.
\newblock A commutator theorem and weighted {BMO}.
\newblock {\em Trans. Amer. Math. Soc.}, 292(1):103--122, 1985.

\bibitem{MR3716205}
L.~Chaffee and D.~Cruz-Uribe.
\newblock Necessary conditions for the boundedness of linear and bilinear
  commutators on {B}anach function spaces.
\newblock {\em Math. Inequal. Appl.}, 21(1):1--16, 2018.

\bibitem{MR2869172}
D.~Chung, M.C. Pereyra, and C.~P\'erez.
\newblock Sharp bounds for general commutators on weighted {L}ebesgue spaces.
\newblock {\em Trans. Amer. Math. Soc.}, 364(3):1163--1177, 2012.

\bibitem{CRW}
R.R. Coifman, R.~Rochberg, and G.~Weiss.
\newblock Factorization theorems for {H}ardy spaces in several variables.
\newblock {\em Ann. of Math. (2)}, 103(3):611--635, 1976.

\bibitem{CMP07}
D.~Cruz-Uribe, J.M. Martell, and C.~P\'{e}rez.
\newblock Sharp two-weight inequalities for singular integrals, with
  applications to the {H}ilbert transform and the {S}arason conjecture.
\newblock {\em Adv. Math.}, 216(2):647--676, 2007.

\bibitem{MR2797562}
D.~Cruz-Uribe, J.M. Martell, and C.~P{\'e}rez.
\newblock {\em Weights, extrapolation and the theory of {R}ubio de {F}rancia},
  volume 215 of {\em Operator Theory: Advances and Applications}.
\newblock Birkh\"auser/Springer Basel AG, Basel, 2011.

\bibitem{CMP09}
D.~Cruz-Uribe, J.M. Martell, and C.~P\'{e}rez.
\newblock Sharp weighted estimates for classical operators.
\newblock {\em Adv. Math.}, 229(1):408--441, 2012.

\bibitem{MR2918187}
D.~Cruz-Uribe and K.~Moen.
\newblock Sharp norm inequalities for commutators of classical operators.
\newblock {\em Publ. Mat.}, 56(1):147--190, 2012.

\bibitem{MR1713140}
D.~Cruz-Uribe and C.~P\'{e}rez.
\newblock Sharp two-weight, weak-type norm inequalities for singular integral
  operators.
\newblock {\em Math. Res. Lett.}, 6(3-4):417--427, 1999.

\bibitem{MR1793688}
D.~Cruz-Uribe and C.~P\'{e}rez.
\newblock Two-weight, weak-type norm inequalities for fractional integrals,
  {C}alder\'{o}n-{Z}ygmund operators and commutators.
\newblock {\em Indiana Univ. Math. J.}, 49(2):697--721, 2000.

\bibitem{cruz-uribe-perez02}
D.~Cruz-Uribe and C.~P{\'e}rez.
\newblock On the two-weight problem for singular integral operators.
\newblock {\em Ann. Sc. Norm. Super. Pisa Cl. Sci. (5)}, 1(4):821--849, 2002.

\bibitem{MR3167497}
D.~Cruz-Uribe, A.~Reznikov, and A.~Volberg.
\newblock Logarithmic bump conditions and the two-weight boundedness of
  {C}alder\'{o}n-{Z}ygmund operators.
\newblock {\em Adv. Math.}, 255:706--729, 2014.

\bibitem{MR3451366}
I.~Holmes, M.T. Lacey, and B.~Wick.
\newblock Bloom's inequality: commutators in a two-weight setting.
\newblock {\em Arch. Math. (Basel)}, 106(1):53--63, 2016.

\bibitem{MR3606434}
I.~Holmes, M.T. Lacey, and B.~Wick.
\newblock Commutators in the two-weight setting.
\newblock {\em Math. Ann.}, 367(1-2):51--80, 2017.

\bibitem{Hy1}
T.P Hyt\"onen.
\newblock Of commutators and jacobians.
\newblock {\em Preprint, \textsf{arXiv:2006.11896}}, 2019.

\bibitem{HyLiOi}
T.P Hyt\"onen, K.~Li, and T.~Oikari.
\newblock Iterated commutators under a joint condition on the tuple of
  multiplying functions.
\newblock {\em Preprint, \textsf{arXiv:1910.00364}}, 2020.

\bibitem{MR4050113}
G.H. Iba\~{n}ez Firnkorn and I.P. Rivera-R\'{\i}os.
\newblock Sparse and weighted estimates for generalized {H}\"{o}rmander
  operators and commutators.
\newblock {\em Monatsh. Math.}, 191(1):125--173, 2020.

\bibitem{IsPoTr}
J.~Isralowitz, S.~Pott, and S.~Treil.
\newblock {Commutators in the two scalar and matrix weighted setting}.
\newblock {\em preprint}, 2017.
\newblock \textsf{arXiv:2001.11182}.

\bibitem{MR1239426}
R.L. Johnson and C.J. Neugebauer.
\newblock Properties of {BMO} functions whose reciprocals are also {BMO}.
\newblock {\em Z. Anal. Anwendungen}, 12(1):3--11, 1993.

\bibitem{Lac1}
M.T. Lacey.
\newblock Two-weight inequality for the {H}ilbert transform: a real variable
  characterization, {II}.
\newblock {\em Duke Math. J.}, 163(15):2821--2840, 2014.

\bibitem{MR3532130}
M.T. Lacey.
\newblock On the separated bumps conjecture for {C}alder\'{o}n-{Z}ygmund
  operators.
\newblock {\em Hokkaido Math. J.}, 45(2):223--242, 2016.

\bibitem{MR3285857}
M.T. Lacey, E.T. Sawyer, C.-Y. Shen, and I.~Uriarte-Tuero.
\newblock Two-weight inequality for the {H}ilbert transform: a real variable
  characterization, {I}.
\newblock {\em Duke Math. J.}, 163(15):2795--2820, 2014.

\bibitem{Lerner20}
A.K. Lerner.
\newblock On separated bump conditions for calder\'{o}n-zygmund operators.
\newblock {\em preprint}, 2020.
\newblock \textsf{arXiv:2008.05866 }.

\bibitem{MR3695871}
A.K. Lerner, S.~Ombrosi, and I.P. Rivera-R\'{\i}os.
\newblock On pointwise and weighted estimates for commutators of
  {C}alder\'{o}n-{Z}ygmund operators.
\newblock {\em Adv. Math.}, 319:153--181, 2017.

\bibitem{MR3919564}
A.K. Lerner, S.~Ombrosi, and I.P. Rivera-R\'{\i}os.
\newblock Commutators of singular integrals revisited.
\newblock {\em Bull. Lond. Math. Soc.}, 51(1):107--119, 2019.

\bibitem{LORR}
A.K. Lerner, S.~Ombrosi, and I.P. Riveria-R\'ios.
\newblock On two weight estimates for iterated commutators.
\newblock {\em Preprint, \textsf{arXiv:2006.11896}}, 2020.

\bibitem{MR3127385}
F.~Nazarov, A.~Reznikov, S.~Treil, and A.~Volberg.
\newblock A {B}ellman function proof of the {$L^2$} bump conjecture.
\newblock {\em J. Anal. Math.}, 121:255--277, 2013.

\bibitem{perez94}
C.~P{\'e}rez.
\newblock Two weighted inequalities for potential and fractional type maximal
  operators.
\newblock {\em Indiana Univ. Math. J.}, 43(2):663--683, 1994.

\bibitem{MR1317714}
C.~P\'{e}rez.
\newblock Endpoint estimates for commutators of singular integral operators.
\newblock {\em J. Funct. Anal.}, 128(1):163--185, 1995.

\bibitem{MR1481632}
C.~P\'{e}rez.
\newblock Sharp estimates for commutators of singular integrals via iterations
  of the {H}ardy-{L}ittlewood maximal function.
\newblock {\em J. Fourier Anal. Appl.}, 3(6):743--756, 1997.

\bibitem{Saw1}
E.T. Sawyer.
\newblock A characterization of a two-weight norm inequality for maximal
  operators.
\newblock {\em Studia Math.}, 75(1):1--11, 1982.

\bibitem{Saw2}
E.T. Sawyer.
\newblock A characterization of two weight norm inequalities for fractional and
  {P}oisson integrals.
\newblock {\em Trans. Amer. Math. Soc.}, 308(2):533--545, 1988.

\end{thebibliography}

\end{document}